\documentclass[11pt]{article}
\usepackage[a4paper,margin=30mm]{geometry}
\usepackage{amsmath,amsthm,amssymb,mathtools}
\usepackage{microtype}
\usepackage{tikz}
\usetikzlibrary{calc}
\usepackage[colorlinks=true,citecolor=black,linkcolor=black,urlcolor=blue]{hyperref}

\hypersetup{
  pdftitle={Root cubes and two-vertex deletions from daisy grids},
  pdfauthor={Guangfu Wang}
}

\newcommand{\cO}{\mathcal{O}}
\newcommand{\cR}{\mathcal{R}}
\newcommand{\cube}{\{0,1\}}
\newcommand{\dn}{\downarrow}

\theoremstyle{plain}
\newtheorem{theorem}{Theorem}
\newtheorem{lemma}[theorem]{Lemma}
\newtheorem{corollary}[theorem]{Corollary}

\theoremstyle{definition}
\newtheorem{definition}[theorem]{Definition}

\theoremstyle{remark}
\newtheorem{remark}[theorem]{Remark}

\newcommand{\email}[1]{\texttt{#1}}
\newcommand{\arxiv}[1]{\href{https://arxiv.org/abs/#1}{\texttt{arXiv:#1}}}
\makeatletter
\def\doi{\begingroup\catcode`\_=12\relax\doi@arg}
\def\doi@arg#1{\href{https://doi.org/#1}{\nolinkurl{doi:#1}}\endgroup}
\makeatother

\title{Root cubes and two-vertex deletions from daisy grids}

\author{Guangfu Wang\\[1ex]
\small School of Mathematics and Information Sciences, Yantai University,\\
\small Yantai, Shandong 264005, China\\
\small \email{gfwang@ytu.edu.cn}}
\date{}

\begin{document}
\maketitle

\begin{abstract}
Daisy cubes are finite partial cubes that admit an isometric hypercube
embedding whose labels form a Boolean down-set.  Call a vertex a root if it
can receive the all-zero label in such an embedding.  Using the peripheral
$\Theta$-class characterization implicit in earlier work of Taranenko and of
Xie--Xu, and Vesel's description of possible zero vertices, we prove the
following intrinsic refinement.  If exactly $b$ classes are peripheral on
both sides, then $G\cong G_0\square Q_b$, where $G_0$ has a unique root and
the roots of $G$ induce the displayed convex $b$-cube.

Our main results concern the daisy grids
$B_{r,s}=P_3^{\square r}\square Q_s$.  We classify every nonempty graph
$B_{r,s}-\{x,y\}$ that is a partial cube, a daisy cube, or a minimal
forbidden partial-cube minor for daisy cubes.  The ambient-isometric cases
admit a uniform coordinate description.  A local four-cycle argument reduces
every non-isometric noncorner partial-cube deletion to $r+s\leq3$, and an
exact orbit analysis leaves precisely eleven sporadic orbits.  None is daisy
or pc-minor-minimal.  Consequently, a deletion is pc-minor-minimal non-daisy
exactly when the deleted vertices are opposite corners with $s\geq1$ and
$2r+s\geq3$, or when a boundary $P_3$-edge is deleted from $B_{1,1}$.  This
recovers the known hypercube and one-$P_3$ examples and produces a new
infinite family with at least two $P_3$-factors.  A dependency-free program
provides exact certificates for the finite orbit analysis.
\end{abstract}

\noindent\textbf{Mathematics Subject Classifications:}
05C12, 05C75, 05C83.

\smallskip
\noindent\textbf{Keywords:}
Daisy cube; partial cube; Djokovi\'c--Winkler relation; peripheral
$\Theta$-class; Cartesian product; partial cube-minor; vertex deletion.

\section{Introduction}
All graphs considered here are finite, simple, and connected, unless stated otherwise. For a graph $G$, let $V(G)$ and $E(G)$ denote its vertex set and edge set. The distance between $u,v\in V(G)$ is denoted by $d_G(u,v)$, or by $d(u,v)$ when the graph is clear. If $S\subseteq V(G)$, then $G[S]$ is the subgraph induced by $S$. The path on $k$ vertices is denoted by $P_k$, and $K_1$ denotes the one-vertex graph.

Partial cubes are graphs whose shortest-path metric is realizable as Hamming distance on binary strings. Equivalently, they are the graphs that occur as isometric subgraphs of hypercubes. Their basic coordinate structure is encoded by the Djokovi\'c--Winkler relation: in a partial cube, this relation partitions the edge set into equivalence classes, called $\Theta$-classes. Each $\Theta$-class behaves as one hypercube coordinate and determines two complementary convex halfspaces \cite{Djokovic1973,Winkler1984,ImrichKlavzar2000}. This cut structure is one of the main reasons why partial cubes admit a rich theory of convexity, Cartesian products, median-type phenomena, and minor operations adapted to isometric hypercube embeddings.

Daisy cubes were introduced by Klav\v{z}ar and Mollard \cite{KlavzarMollard2019}. For $X\subseteq\cube^n$, the daisy cube generated by $X$ is the subgraph of $Q_n$ induced by the union of all intervals from $0^n$ to vertices of $X$; equivalently, its vertex set is $\dn X$, the coordinatewise down-set generated by $X$. Daisy cubes therefore form a natural down-set subclass of partial cubes. They include hypercubes and several familiar hypercube-like families, such as Fibonacci cubes, Lucas cubes, and bipartite wheels \cite{Klavzar2013,MunariniPerelliCippoZagagliaSalvi2001,KlavzarMollard2019}.

The theory has developed in several directions. Taranenko \cite{Taranenko2020} characterized daisy cubes by a special expansion procedure and introduced daisy graphs of rooted graphs. Vesel \cite{Vesel2021} characterized the vertices that may receive the all-zero label in a proper embedding and obtained a linear-time proper-embedding algorithm for graphs known to be daisy cubes. Dravec and Taranenko \cite{DravecTaranenko2023} extended the theme to daisy Hamming graphs. Polynomial characterizations and identities for daisy cubes have also been investigated, most recently by Zheng, Xie, and Xu \cite{ZhengXieXu2026}. In chemical graph theory, daisy cubes occur as resonance graphs of suitable plane bipartite graphs \cite{BrezovnikCheTratnikZigert2025}, with subsequent work on decomposition and $\tau$-graph structure \cite{CheChen2025,CheTratnikZigert2026}.

The down-set definition is useful, but it is not intrinsic: it requires a
hypercube embedding, a choice of $0^n$, and an orientation of every
coordinate.  Taranenko \cite[Proposition~2.1]{Taranenko2020} proved that every
$\Theta$-class of a daisy cube is peripheral.  Xie and Xu
\cite[Theorem~1.2 and the proof of (iv)$\Rightarrow$(v)]{XieXu2025} proved the
converse for finite median graphs; their coordinate-orientation argument in
fact uses only the partial-cube structure.  We use this general form as a
previously implicit characterization, not as a new contribution.

We first formulate an intrinsic refinement of the root structure.
Vesel \cite[Theorem~1 and Lemma~2]{Vesel2021} showed that the vertices that can
receive the zero label form a Boolean interval.  Taranenko
\cite[Proposition~3.2]{Taranenko2020} identified the corresponding one-coordinate
$K_2$ factor.  We identify the free coordinates intrinsically as the classes
that are peripheral on both sides and extract all of them simultaneously.  If
their number is $b$, then
\[
 G\cong G_0\square Q_b,
\]
where $G_0$ has a unique daisy root.  The roots induce the displayed convex
$Q_b$-factor, and the equivalence classes of minimum-dimensional proper
embeddings are in bijection with them.  The purpose of this formulation is to
combine the earlier root results into a canonical simultaneous factorization
with an intrinsic halfspace description.

Since daisy cubes are closed under restrictions to halfspaces and
contractions of $\Theta$-classes, they admit a family of minimal forbidden
partial cube-minors. Xie and Xu \cite[Section~4]{XieXu2025} recorded the
families $Q_n^{-+}$, $Q_n^{--}$, and $H_n^{--}$ and asked for a forbidden
pc-minor characterization. Merely defining the set of all minimal
non-daisy pc-minors does not enumerate it; a useful step is therefore to
classify natural, sufficiently broad sources of candidate obstructions.

The Cartesian product of graphs $G$ and $H$ is denoted by $G\square H$
\cite{HammackImrichKlavzar2011,ImrichKlavzar2000}.  We write
$G^{\square r}$ for the $r$-fold product and use $G^{\square0}=K_1$.

The principal contribution is the following complete deletion analysis for
\[
 B_{r,s}=P_3^{\square r}\square Q_s.
\]
For arbitrary distinct vertices $x,y$, we:
\begin{enumerate}
\item characterize exactly when $B_{r,s}-\{x,y\}$ is isometric in the
      ambient product;
\item classify all corner deletions that remain partial cubes, including the
      unique non-isometric parameter pattern;
\item prove that every non-isometric noncorner partial-cube deletion satisfies
      $r+s\leq3$, reducing an infinite problem to forty automorphism orbits;
\item certify that precisely eleven of these orbits are partial cubes and that
      none is daisy or pc-minor-minimal; and
\item derive the complete deletion classification and exact counting formulas.
\end{enumerate}
The result contains the previously known graphs $Q_s^{--}$,
$H_n^{--}$, and $P_4$ from Xie and Xu \cite[Section~4]{XieXu2025}; its
higher-$P_3$ branch is new.  Section~2 fixes notation, Section~3 gives the
root-cube refinement, Section~4 proves the deletion classification, and
Appendix~A records the reproducible finite verification.

\section{Notation and preliminaries}

For $n\ge0$, the hypercube $Q_n$ has vertex set $\cube^n$; two binary strings are adjacent if they differ in exactly one coordinate. We set $Q_0=K_1$. A subgraph $H$ of a graph $G$ is \emph{isometric} if $d_H(x,y)=d_G(x,y)$ for all $x,y\in V(H)$. A \emph{partial cube} is a graph isomorphic to an isometric subgraph of a hypercube.

For edges $ab,xy\in E(G)$ of a connected graph $G$, the Djokovi\'c--Winkler relation is
\[
 ab\,\Theta_G\,xy
 \quad\Longleftrightarrow\quad
 d(a,x)+d(b,y)\ne d(a,y)+d(b,x).
\]
When no confusion is possible, we write $\Theta$ instead of $\Theta_G$. A connected bipartite graph is a partial cube if and only if $\Theta$ is an equivalence relation on its edge set \cite{Djokovic1973,Winkler1984}. The equivalence classes are called $\Theta$-classes. In a partial cube, the number of $\Theta$-classes is the isometric dimension, that is, the dimension of the standard hypercube embedding.

Let $G$ be a partial cube and let $ab\in E(G)$. Define
\begin{align*}
 W_{ab}&=\{x\in V(G):d(a,x)<d(b,x)\},\\
 W_{ba}&=\{x\in V(G):d(b,x)<d(a,x)\}.
\end{align*}
These sets are complementary convex halfspaces of $G$. The $\Theta$-class containing $ab$ is
\[
 F_{ab}=\{xy\in E(G): x\in W_{ab}\text{ and }y\in W_{ba}\}.
\]
For the same oriented edge $ab$, put
\[
 U_{ab}=\{x\in W_{ab}:x\text{ is incident with an edge of }F_{ab}\},
\]
and define $U_{ba}$ analogously. Equivalently, if $\theta$ is a $\Theta$-class and $W$ is one of its two halfspaces, then
\[
 U(\theta,W)=\{x\in W:x\text{ is incident with an edge of }\theta\}.
\]
Thus $U_{ab}=U(F_{ab},W_{ab})$.

\begin{definition}[\cite{Taranenko2020,XieXu2025}]\label{def:peripheral}
A $\Theta$-class $\theta$ of a partial cube is \emph{peripheral} if $U(\theta,W)=W$ for at least one of the two halfspaces $W$ determined by $\theta$.
\end{definition}

Here the definition is used for arbitrary partial cubes, without a
median-graph assumption.

For a partial cube $G$ with $m$ $\Theta$-classes, choose for the $i$th class one halfspace $W_i^0$ and denote the other one by $W_i^1$. The standard $\Theta$-coordinate label of a vertex $v$ is the binary string $\lambda(v)\in\cube^m$ given by
\[
 \lambda(v)_i=\varepsilon
 \quad\Longleftrightarrow\quad
 v\in W_i^\varepsilon,
 \qquad \varepsilon\in\{0,1\}.
\]
The map $\lambda$ is an isometric embedding into $Q_m$. Moreover, an edge of $G$ belongs to the $i$th $\Theta$-class exactly when the labels of its endpoints differ only in coordinate $i$.

The coordinatewise order on $\cube^n$ is denoted by $\le$. Thus, for $x=(x_1,\ldots,x_n)$ and $y=(y_1,\ldots,y_n)$, we have $y\le x$ if $y_i\le x_i$ for all $i$. For $X\subseteq\cube^n$, write
\[
 \dn X=\{y\in\cube^n:y\le x\text{ for some }x\in X\}.
\]
A nonempty set $D\subseteq\cube^n$ is a \emph{down-set} if $D=\dn D$.

\begin{definition}[\cite{KlavzarMollard2019}]\label{def:daisy}
A graph is a \emph{daisy cube} if it is isomorphic to $Q_n[D]$ for some $n\ge0$ and some nonempty down-set $D\subseteq\cube^n$. Equivalently, it is isomorphic to $Q_n[\dn X]$ for some $X\subseteq\cube^n$.
\end{definition}

Let $G$ be a daisy cube with $m$ $\Theta$-classes. A vertex $r\in V(G)$ is a
\emph{daisy root} if there is an isometric embedding
$\lambda:V(G)\to\cube^m$ such that $\lambda(V(G))$ is a down-set and
$\lambda(r)=0^m$. We denote the set of daisy roots by $\cR(G)$. These are
the minimal vertices associated with proper embeddings in the terminology
of Vesel \cite{Vesel2021}. We use the isometric dimension $m$ in this
definition, so inactive coordinates are excluded.

A \emph{restriction} of a partial cube $G$ with respect to a $\Theta$-class is the subgraph induced by one of the two halfspaces of that class. A \emph{contraction} of $G$ with respect to a $\Theta$-class is obtained by contracting all edges of that class and then suppressing loops and parallel edges. A \emph{partial cube-minor}, abbreviated \emph{pc-minor}, is any graph obtained by a finite sequence of restrictions and contractions. Such a pc-minor is \emph{proper} if it is not isomorphic to the original graph. Restrictions and contractions preserve partial cubes \cite{Marc2018,ChepoiKnauerMarc2020}, and daisy cubes are pc-minor-closed \cite{KlavzarMollard2019,Taranenko2020,XieXu2025}.

The following standard inheritance property will be used repeatedly; see,
for example, \cite{Djokovic1973,ImrichKlavzar2000}.

\begin{lemma}\label{lem:isometric-theta}
Let $H$ be an isometric subgraph of a partial cube $G$. If $e,f\in E(H)$, then
\[
 e\,\Theta_H\,f\quad\Longleftrightarrow\quad e\,\Theta_G\,f.
\]
Consequently, each $\Theta$-class of $H$ is the nonempty intersection of $E(H)$ with a $\Theta$-class of $G$, and the halfspaces of $H$ are obtained by intersecting the corresponding halfspaces of $G$ with $V(H)$.
\end{lemma}

\section{Peripheral classes and the root-cube decomposition}

We begin by isolating the exact previously known ingredient.  The forward
implication below is Taranenko's Proposition~2.1
\cite{Taranenko2020}.  The reverse implication is the coordinate-orientation
argument in the proof of Xie and Xu's Theorem~1.2(iv)$\Rightarrow$(v)
\cite{XieXu2025}; that argument does not use medianity and therefore applies
verbatim to every finite partial cube.  We therefore quote the resulting
characterization without repeating its proof.

\begin{theorem}\label{thm:peripheral-characterization}
A finite partial cube is a daisy cube if and only if every $\Theta$-class is peripheral.
\end{theorem}

For a $\Theta$-class $\theta$ and $r\in V(G)$, write $W_\theta(r)$ for the
halfspace of $\theta$ containing $r$ and $\overline W_\theta(r)$ for the
opposite halfspace.

Vesel \cite[Theorem~1 and Lemma~2]{Vesel2021} showed, in the language of a
fixed proper embedding and its maximal labels, that the minimal vertices form
a Boolean interval and yield equivalent proper embeddings.  Taranenko
\cite[Proposition~3.2]{Taranenko2020} observed the associated $K_2$ factor for
one two-sided peripheral class.  We now give an intrinsic root test and
extract all such factors simultaneously.

\begin{theorem}[Root criterion]\label{thm:root-criterion}
Let $G$ be a finite daisy cube and $r\in V(G)$. Then
\[
 r\in\cR(G)
 \quad\Longleftrightarrow\quad
 U\bigl(\theta,\overline W_\theta(r)\bigr)
   =\overline W_\theta(r)
 \quad\text{for every $\Theta$-class $\theta$.}
\]
\end{theorem}

\begin{proof}
Suppose that $r$ is a daisy root, and choose a proper $\Theta$-coordinate
embedding with $r$ labelled $0^m$. For every coordinate class, the
halfspace opposite $r$ is the $1$-side. Downward closure shows that every
vertex on this side has its coordinate-lowering neighbor, so the displayed
equality holds.

Conversely, orient $\overline W_\theta(r)$ as the $1$-side for every
$\theta$.  The assumed equalities are precisely the coordinate-lowering
condition in the argument cited before
Theorem~\ref{thm:peripheral-characterization}.  Hence the resulting standard
embedding has a down-set as its image. Since
$r$ lies on every $0$-side, its label is $0^m$, and therefore
$r\in\cR(G)$.
\end{proof}

A $\Theta$-class is \emph{two-sided peripheral} if both of its halfspaces
are equal to their corresponding $U$-sets. The set of all such classes is
denoted by $\mathcal B(G)$.

\begin{theorem}[Root-cube factorization]\label{thm:root-cube}
Let $G$ be a finite daisy cube and put $b=|\mathcal B(G)|$. Then there is
a daisy cube $G_0$, unique up to isomorphism, such that
\[
 G\cong G_0\square Q_b,
\]
where $G_0$ has no two-sided peripheral $\Theta$-class and has a unique
daisy root $r_0$. Under this decomposition,
\[
 \cR(G)=\{r_0\}\times V(Q_b).
\]
In particular, $G[\cR(G)]$ is a convex $b$-cube and
$|\cR(G)|=2^b$.
\end{theorem}

\begin{proof}
Let $\theta_1,\ldots,\theta_m$ be the $\Theta$-classes. Orient a peripheral
halfspace of each class as its $1$-side and let
$D\subseteq\cube^m$ be the image of the standard embedding. By
Theorem~\ref{thm:peripheral-characterization}, $D$ is a down-set. Reorder
the coordinates so that the last $b$ coordinates correspond to
$\mathcal B(G)$.

If $i$ is one of these last coordinates and $x\in D$, then
$x\oplus e_i\in D$. Indeed, when $x_i=1$ this follows from peripheralness
of the $1$-side, and when $x_i=0$ it follows from peripheralness of the
$0$-side. Thus $D$ is invariant under independently flipping the last
$b$ coordinates. If $D_0$ is its projection onto the first $m-b$
coordinates, then
\[
 D=D_0\times\cube^b.
\]
The set $D_0$ is a down-set, and hence
$G\cong Q_{m-b}[D_0]\square Q_b$. Put $G_0=Q_{m-b}[D_0]$.

A remaining coordinate is two-sided peripheral in $G_0$ if and only if
the corresponding coordinate is two-sided peripheral in
$G_0\square Q_b$. Hence $G_0$ has no such class. Every $\Theta$-class of
$G_0$ consequently has exactly one peripheral side. By
Theorem~\ref{thm:root-criterion}, a daisy root must lie on the opposite
side of each class, so all coordinate orientations are forced. Existence
follows from the down-set representation, and injectivity of the standard
embedding gives a unique root $r_0$.

For the $b$ two-sided classes either orientation is allowed, independently,
whereas all other orientations are forced by $r_0$. The root criterion
therefore gives
$\cR(G)=\{r_0\}\times V(Q_b)$. This set induces a convex $Q_b$-fiber.
Finally, $\mathcal B(G)$ is intrinsic and $G_0$ is obtained by contracting
all its classes, so $G_0$ is unique up to isomorphism.
\end{proof}

\begin{corollary}\label{cor:embedding-count}
If $b=|\mathcal B(G)|$, then a finite daisy cube $G$ has exactly $2^b$
equivalence classes of minimum-dimensional proper embeddings, where
embeddings that differ only by a permutation of coordinates are identified.
\end{corollary}

\begin{proof}
The all-zero vertex of a proper embedding is a daisy root. Conversely,
Theorem~\ref{thm:root-criterion} constructs a proper embedding from every
root. Once the root is fixed, each $\Theta$-class determines its coordinate
and its orientation, leaving only a permutation of coordinates. Apply
Theorem~\ref{thm:root-cube}; compare Vesel
\cite[Theorem~1]{Vesel2021}.
\end{proof}

\begin{corollary}[Certificate form]\label{cor:certificate}
For a finite partial cube with its $\Theta$-classes and halfspaces given,
exactly one of the following certificates exists:
\begin{enumerate}
\item a choice of a peripheral side for every class, which produces a
      proper down-set embedding;
\item a class $\theta$ and vertices
      $x_0\in W_0\setminus U(\theta,W_0)$ and
      $x_1\in W_1\setminus U(\theta,W_1)$, which certify that the graph is
      not a daisy cube.
\end{enumerate}
\end{corollary}

\begin{proof}
This is Theorem~\ref{thm:peripheral-characterization}, with a failed
peripheral test witnessed on both sides of one class.
\end{proof}

Let $\cO_{\mathrm D}$ be the set of all partial cubes $H$ such that $H$ is not a daisy cube and every proper pc-minor of $H$ is a daisy cube. These graphs are the minimal forbidden pc-minors for the class of daisy cubes.

\begin{corollary}\label{cor:obstruction-set}
For a finite partial cube $G$,
\[
\begin{aligned}
 G\text{ is a daisy cube}
 \quad\Longleftrightarrow\quad
 &G\text{ has no pc-minor isomorphic}\\
 &\text{to a member of }\cO_{\mathrm D}.
\end{aligned}
\]
Moreover, $\cO_{\mathrm D}$ is exactly the set of pc-minor-minimal partial cubes that contain a non-peripheral $\Theta$-class.
\end{corollary}

\begin{proof}
Because daisy cubes are pc-minor-closed, a daisy cube cannot have a pc-minor in $\cO_{\mathrm D}$. Conversely, if $G$ is not a daisy cube, then among all non-daisy pc-minors of $G$ choose one with minimum $|V|+|E|$. This chosen pc-minor belongs to $\cO_{\mathrm D}$. It follows directly from Theorem~\ref{thm:peripheral-characterization} that, within finite partial cubes, being non-daisy is equivalent to having at least one non-peripheral $\Theta$-class.
\end{proof}

\begin{remark}\label{rem:enumeration}
Corollary~\ref{cor:obstruction-set} is a formal obstruction-set
reformulation, not an enumeration of $\cO_{\mathrm D}$. A full forbidden
pc-minor theorem still requires the classification of pc-minor-minimal
partial cubes with a non-peripheral $\Theta$-class. Section~4 gives such a
classification inside a broad two-vertex-deletion model.
\end{remark}

\section{Two-vertex deletions from daisy grids}

For integers \(r,s\ge0\), put
\[
 B_{r,s}=P_3^{\square r}\square Q_s.
\]
We use the coordinate model
\[
 V(B_{r,s})=\{0,1,2\}^r\times\{0,1\}^s,
\]
where two vertices are adjacent if exactly one coordinate changes by \(1\).
A \emph{corner} has every \(P_3\)-coordinate in \(\{0,2\}\); the remaining
coordinates are already binary. For arbitrary distinct vertices $x,y$, put
\[
 H_{r,s}(x,y)=B_{r,s}-\{x,y\},
\]
and define
\begin{align*}
 \mu(x,y)&=\bigl|\{i\in[r]:x_i=y_i=1\}\bigr|,\\
 \alpha(x,y)&=\bigl|\{i\in[r]:|x_i-y_i|=1\}\bigr|,\\
 p(x,y)&=\bigl|\{i\in[r]:\{x_i,y_i\}=\{0,2\}\}\bigr|,\\
 q(x,y)&=\bigl|\{j\in[s]:x_{r+j}\ne y_{r+j}\}\bigr|.
\end{align*}
Thus $x,y$ are both corners exactly when $\mu(x,y)=\alpha(x,y)=0$.
For corners, $p$ and $q$ record the numbers of differing $P_3$- and
$P_2$-factors, and the corners are opposite exactly when $p=r$ and $q=s$.
We call $\{x,y\}$ a \emph{boundary $P_3$-edge} if $xy$ is an edge in a
$P_3$-factor and one endpoint is a corner of $B_{r,s}$. Equivalently,
$\mu=0$, $\alpha=1$, $p=q=0$, and all other $P_3$-coordinates are common
endpoints.

Figure~\ref{fig:b11-p4} shows the two occurrences of the smallest obstruction
$P_4$ inside $B_{1,1}$.  It also fixes the geometric conventions used in the
later deletion diagrams.

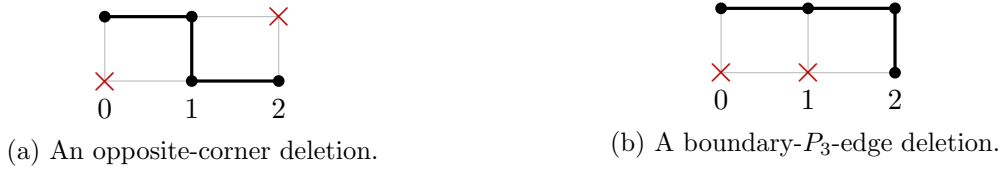
\begin{figure}[t]
\centering
\begin{minipage}{.46\textwidth}
\centering
\begin{tikzpicture}[x=1.15cm,y=.85cm]
  \foreach \j in {0,1}{\draw[gray!55] (0,\j)--(2,\j);}
  \foreach \i in {0,1,2}{\draw[gray!55] (\i,0)--(\i,1);}
  \draw[very thick] (2,0)--(1,0)--(1,1)--(0,1);
  \foreach \p in {(1,0),(2,0),(0,1),(1,1)}
    \fill \p circle (2.2pt);
  \node[red!75!black,font=\Large] at (0,0) {$\times$};
  \node[red!75!black,font=\Large] at (2,1) {$\times$};
  \node[below] at (0,-.12) {$0$};
  \node[below] at (1,-.12) {$1$};
  \node[below] at (2,-.12) {$2$};
\end{tikzpicture}

\small (a) An opposite-corner deletion.
\end{minipage}\hfill
\begin{minipage}{.46\textwidth}
\centering
\begin{tikzpicture}[x=1.15cm,y=.85cm]
  \foreach \j in {0,1}{\draw[gray!55] (0,\j)--(2,\j);}
  \foreach \i in {0,1,2}{\draw[gray!55] (\i,0)--(\i,1);}
  \draw[very thick] (0,1)--(1,1)--(2,1)--(2,0);
  \foreach \p in {(2,0),(0,1),(1,1),(2,1)}
    \fill \p circle (2.2pt);
  \node[red!75!black,font=\Large] at (0,0) {$\times$};
  \node[red!75!black,font=\Large] at (1,0) {$\times$};
  \node[below] at (0,-.12) {$0$};
  \node[below] at (1,-.12) {$1$};
  \node[below] at (2,-.12) {$2$};
\end{tikzpicture}

\small (b) A boundary-$P_3$-edge deletion.
\end{minipage}
\caption{Two deletions from $B_{1,1}=P_3\square Q_1$ that yield $P_4$.
Crosses denote deleted vertices; thick edges form the surviving graph.}
\label{fig:b11-p4}
\end{figure}

We begin with the daisy property of the ambient products and their
one-corner deletions.

\begin{lemma}[Products and one-corner deletions]\label{lem:products-daisy}
For all \(r,s\ge0\), the graph \(B_{r,s}\) is a daisy cube. If one corner
is deleted and the remaining graph is nonempty, the resulting graph is
also a daisy cube.
\end{lemma}

\begin{proof}
The first assertion follows from the Cartesian-product characterization of
daisy cubes \cite[Theorem~4.1 and Corollary~4.2]{BrezovnikCheTratnikZigert2025},
because both $P_3$ and $Q_s$ are daisy cubes.

Let \(c\) be a corner. In every \(P_3\)-factor, choose the above embedding
of $P_3$ as the down-set $\{00,10,01\}$ so that the endpoint occurring in
$c$ receives a maximal label.  Orient every $Q_s$-coordinate so that the bit
of $c$ is $1$.
Then \(c\) is a maximal element of the product down-set. Deleting a maximal
element preserves downward closure.
\end{proof}

We next record the product-geodesic fact used in the sufficiency proof.
Imrich, Zmazek, and \v{Z}erovnik
\cite[p.~276]{ImrichZmazekZerovnik2003} proved the stronger statement that
vertices differing in $k$ Cartesian coordinates admit at least $k$
internally vertex-disjoint geodesics.  We need only the following special
case.

\begin{lemma}[Cyclic block geodesics]\label{lem:cyclic-geodesics}
Let
\[
 I=P_{\ell_1+1}\square\cdots\square P_{\ell_k+1},
 \qquad k\geq3,\quad \ell_i\geq1,
\]
and let $u=(0,\ldots,0)$ and $v=(\ell_1,\ldots,\ell_k)$ be opposite
corners.  Then $I$ contains three $u,v$-geodesics whose internal vertex
sets are pairwise disjoint.
\end{lemma}

We now characterize ambient isometry for arbitrary deleted vertices.  The
second alternative below is the only noncorner case.

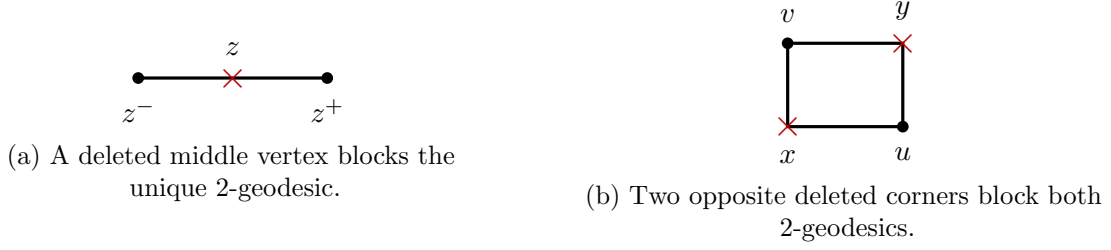
\begin{figure}[t]
\centering
\begin{minipage}{.43\textwidth}
\centering
\begin{tikzpicture}[x=1.25cm,y=1cm]
  \draw[very thick] (0,0)--(1,0)--(2,0);
  \fill (0,0) circle (2.2pt) node[below=4pt] {$z^-$};
  \fill (2,0) circle (2.2pt) node[below=4pt] {$z^+$};
  \node[red!75!black,font=\Large] at (1,0) {$\times$};
  \node[above=5pt] at (1,0) {$z$};
\end{tikzpicture}

\small (a) A deleted middle vertex blocks the unique 2-geodesic.
\end{minipage}\hfill
\begin{minipage}{.49\textwidth}
\centering
\begin{tikzpicture}[x=1.05cm,y=1.05cm]
  \draw[very thick] (0,0) rectangle (1.45,1.05);
  \fill (1.45,0) circle (2.2pt) node[below=4pt] {$u$};
  \fill (0,1.05) circle (2.2pt) node[above=4pt] {$v$};
  \node[red!75!black,font=\Large] at (0,0) {$\times$};
  \node[red!75!black,font=\Large] at (1.45,1.05) {$\times$};
  \node[below=5pt] at (0,0) {$x$};
  \node[above=5pt] at (1.45,1.05) {$y$};
\end{tikzpicture}

\small (b) Two opposite deleted corners block both 2-geodesics.
\end{minipage}
\caption{The two local obstructions to ambient isometry used in
Theorem~\ref{thm:ambient-isometry}.}
\label{fig:ambient-obstructions}
\end{figure}

\begin{theorem}[Ambient-isometry classification]
\label{thm:ambient-isometry}
Assume that $H=H_{r,s}(x,y)$ is nonempty. Then $H$ is an isometric
subgraph of $B_{r,s}$ if and only if one of the following holds:
\begin{enumerate}
\item $x,y$ are corners and $(p(x,y),q(x,y))\ne(0,2)$;
\item $\{x,y\}$ is a boundary $P_3$-edge.
\end{enumerate}
\end{theorem}

\begin{proof}
Suppose first that one of $x,y$, say $z$, has value $1$ in a
$P_3$-coordinate $i$.  Let $z^-$ and $z^+$ be its two neighbors obtained
by changing that coordinate to $0$ and $2$.  If the other deleted vertex is
neither $z^-$ nor $z^+$, then both remain in $H$.  Their unique
length-two geodesic in $B_{r,s}$ is $z^-zz^+$, so $H$ is not isometric.
This is the obstruction in Figure~\ref{fig:ambient-obstructions}(a).
If the other deleted vertex is, say, $z^-$, then the deleted pair is a
$P_3$-edge.  Unless $z^-$ is a corner, it has value $1$ in another
$P_3$-coordinate $k$.  Its two neighbors in coordinate $k$ remain, and
their unique length-two geodesic uses $z^-$.  Again $H$ is not isometric.
Consequently, every ambient-isometric deletion is either a pair of corners
or a boundary $P_3$-edge.

For two corners, suppose first that $(p,q)=(0,2)$.  They are opposite
vertices of the square obtained by varying the two differing binary
coordinates.  The other two vertices of this square remain at ambient
distance two, but both of their length-two paths use a deleted corner.
Thus the deletion is not isometric; see
Figure~\ref{fig:ambient-obstructions}(b).

It remains to prove sufficiency.  Take surviving vertices $u,v$ and let
$I=I_{B_{r,s}}(u,v)$.  This interval is a Cartesian product of the
nontrivial coordinate intervals between $u$ and $v$.  Let $k$ be the
number of its nontrivial factors.  If $k\geq3$, Lemma~\ref{lem:cyclic-geodesics}
provides three internally vertex-disjoint $u,v$-geodesics in $I$.  Since
$u,v$ survive, two deleted vertices cannot meet all three paths.

Assume next that $k\leq2$ and that the deleted vertices are corners.  Every
deleted vertex that belongs to $I$ is a corner of $I$.  If $k\leq1$, the
only corners of $I$ are its endpoints $u,v$, so no deleted vertex is
internal to its unique geodesic.  Suppose that $k=2$, and write
\[
 I=P_{\ell_1+1}\square P_{\ell_2+1},\qquad
 u=(0,0),\quad v=(\ell_1,\ell_2).
\]
The two block-order geodesics pass through the other corners
$c_1=(\ell_1,0)$ and $c_2=(0,\ell_2)$, respectively.  If at most one of
$c_1,c_2$ is deleted, the other geodesic survives.  If both are deleted,
then every varying $P_3$-coordinate has its full two-edge interval, because
$c_1,c_2$ are corners of $B_{r,s}$.  Consequently, if
$\ell_1=\ell_2=1$, both factors are binary and the deleted pair has
$(p,q)=(0,2)$, which is excluded.  Otherwise, after interchanging the
factors, assume $\ell_1=2$.  The geodesic that first moves from $(0,0)$ to
$(1,0)$, then traverses the second factor, and finally moves to
$(2,\ell_2)$ avoids both $c_1$ and $c_2$.

Finally, suppose that the deleted pair is a boundary $P_3$-edge
$\{c,c'\}$, where $c$ is its corner endpoint.  If $k\leq1$, neither
deleted vertex can be internal to $I$: the corner $c$, if present, is an
endpoint, while an interval in the distinguished $P_3$-coordinate that
contains $c'$ internally also has $c$ as an endpoint; in every other
coordinate, $c'$ itself has an endpoint value.  If $k=2$, the two
block-order boundary geodesics are internally disjoint.  Hence one avoids
the deleted pair unless both $c,c'$ lie in $I$.  In that event, $c$ is one
of the two corners of $I$ different from $u,v$, and the edge $cc'$ lies on
the boundary geodesic through $c$; the other boundary geodesic avoids both
deleted vertices.  Thus in every case $I-\{x,y\}$ contains a
$u,v$-geodesic, and $H$ is isometric.
\end{proof}

\begin{corollary}[Corner isometry]
\label{cor:corner-isometry}
For distinct corners $a,b$, the nonempty graph $H_{r,s}(a,b)$ is
isometric in $B_{r,s}$ if and only if $(p(a,b),q(a,b))\ne(0,2)$.
\end{corollary}

\begin{lemma}[Boundary-edge profile]\label{lem:boundary-profile}
Let $\{x,y\}$ be a boundary $P_3$-edge.  If $(r,s)=(1,0)$, then
$H_{r,s}(x,y)\cong K_1$.  Otherwise, $H_{r,s}(x,y)$ is a non-daisy
partial cube with exactly one non-peripheral $\Theta$-class.
\end{lemma}

\begin{proof}
By Theorem~\ref{thm:ambient-isometry}, the deletion is isometric.  Reflect
and permute factors so that
\[
 x=(0,0,\ldots,0;0,\ldots,0),\qquad
 y=(1,0,\ldots,0;0,\ldots,0),
\]
where the first coordinate is a $P_3$-coordinate.  If $B_{r,s}=P_3$,
only the vertex $2$ remains.

Assume $r+s\geq2$.  Consider the $\Theta$-class between levels $1$ and
$2$ in the first coordinate.  The vertex $(2,0,\ldots,0;0,\ldots,0)$
has lost its only incident edge of this class.  On the other side, a
surviving vertex at first-coordinate level $0$ is obtained by changing an
additional factor, and such a vertex is not incident with this class.
Hence neither halfspace is peripheral.

The class between levels $0$ and $1$ remains peripheral on its level-$0$
side: the only missing crossing edge has both endpoints deleted.  Every
class in another $P_3$-factor retains its peripheral extreme side, since
the two deleted vertices have a common endpoint in that factor.  Every
binary-coordinate class is peripheral on the side containing the common
bit of the deleted vertices.  Lemma~\ref{lem:isometric-theta} shows that
these are all the classes of the deletion.  The assertion now follows from
Theorem~\ref{thm:peripheral-characterization}.
\end{proof}

\begin{remark}[A first non-isometric noncorner example]
\label{rem:nonisometric-scope}
Failure of ambient isometry does not rule out an isometric embedding into a
different hypercube.  This phenomenon already occurs for a noncorner pair:
deleting the corner $(0,0)$ and the center $(1,1)$ from
$B_{2,0}=P_3\square P_3$ leaves the seven-vertex boundary path $P_7$;
see Figure~\ref{fig:p3-grid-p7}.
Thus Theorem~\ref{thm:ambient-isometry} is an exact classification of
ambient isometry, but abstract partial-cube embeddings require an additional
analysis.  We first settle the corner branch and then return to all noncorner
pairs in Subsection~\ref{subsec:noncorner}.
\end{remark}

\begin{figure}[t]
\centering
\begin{tikzpicture}[x=1.05cm,y=1.05cm]
  \foreach \i in {0,1,2}{
    \draw[gray!45] (\i,0)--(\i,2);
    \draw[gray!45] (0,\i)--(2,\i);
  }
  \draw[very thick]
    (0,1)--(0,2)--(1,2)--(2,2)--(2,1)--(2,0)--(1,0);
  \foreach \p in {(0,1),(0,2),(1,2),(2,2),(2,1),(2,0),(1,0)}
    \fill \p circle (2.2pt);
  \node[red!75!black,font=\Large] at (0,0) {$\times$};
  \node[red!75!black,font=\Large] at (1,1) {$\times$};
  \node[below] at (0,-.12) {$0$};
  \node[below] at (1,-.12) {$1$};
  \node[below] at (2,-.12) {$2$};
  \node[left] at (-.12,0) {$0$};
  \node[left] at (-.12,1) {$1$};
  \node[left] at (-.12,2) {$2$};
\end{tikzpicture}
\caption{Deleting $(0,0)$ and $(1,1)$ from $P_3\square P_3$ destroys
ambient isometry but leaves the thick seven-vertex path.}
\label{fig:p3-grid-p7}
\end{figure}
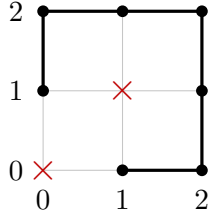

For corner deletions, Corollary~\ref{cor:corner-isometry} leaves only the
parameter pair $(p,q)=(0,2)$, which we now analyze.

\begin{lemma}[The exceptional pair]\label{lem:exceptional-pair}
Suppose that \((p(a,b),q(a,b))=(0,2)\). Then
\(H_{r,s}(a,b)\) is a partial cube if and only if \(r+s=3\). In that case
it is not a daisy cube.
\end{lemma}

\begin{proof}
After reflecting coordinates, write
\[
 B_{r,s}=C\square Q_2,
 \qquad
 a=(c,00),\quad b=(c,11),
 \qquad
 C=B_{r,s-2},
\]
where \(c\) is a corner of \(C\). Put \(d=r+s-2\), the number of Cartesian
factors of \(C\).

If \(d=0\), the two surviving vertices \((c,01)\) and \((c,10)\) are
isolated from each other. Thus the graph is disconnected and is not a
partial cube under our convention.

Suppose that \(d=1\). Then \(C=P_{\ell+1}\), where
\(\ell\in\{1,2\}\), and we may label its vertices
\(0,1,\ldots,\ell\) so that \(c=0\). Set
\(u=(0,01)\), \(v=(0,10)\), and, for \(1\le k\le\ell-1\), put
\(\chi_k(t)=1\) if \(t\ge k+1\) and \(\chi_k(t)=0\) otherwise. The map
\(\varphi:V(H_{r,s}(a,b))\to\cube^{\ell+3}\) defined by
\begin{align*}
 \varphi(u)&=(0,0,0,0,0^{\ell-1}),\\
 \varphi(v)&=(1,1,1,1,0^{\ell-1}),\\
 \varphi(t,x,y)&=(1,0,x,1-y,
                  \chi_1(t),\ldots,\chi_{\ell-1}(t))
                  \quad (1\le t\le\ell)
\end{align*}
is isometric. Indeed, among vertices with positive first coordinate its
Hamming distance is \(|t-t'|+|x-x'|+|y-y'|\), while
\begin{align*}
 d_H(u,(t,x,y))&=t+x+(1-y),\\
 d_H(v,(t,x,y))&=t+(1-x)+y,\\
 d_H(u,v)&=4,
\end{align*}
and these are exactly the corresponding Hamming distances. Hence \(H\) is
a partial cube. The coordinate class that changes \(x\) has the
nonincident vertex \(u\) on its \(0\)-side and the nonincident vertex
\(v\) on its \(1\)-side. This class is not peripheral, so
Theorem~\ref{thm:peripheral-characterization} shows that \(H\) is not a
daisy cube.

Finally, suppose that \(d\ge2\). Choose neighbors \(u_0,v_0\) of \(c\)
in two distinct factors of \(C\), and let \(w_0\) be their common neighbor
in the resulting Cartesian square. In \(H\), consider the edges
\begin{align*}
 e&=(c,01)(v_0,01),\\
 f&=(c,10)(u_0,10),\\
 g&=(v_0,00)(w_0,00).
\end{align*}
For the ordered endpoints
\[
\begin{aligned}
 A&=(c,01), &B&=(v_0,01), &X&=(c,10),\\
 Y&=(u_0,10), &P&=(v_0,00), &Q&=(w_0,00),
\end{aligned}
\]
the relevant distances in \(H\) are
\begin{align*}
 &d(A,X)=d(B,Y)=4,
 &&d(A,Y)=d(B,X)=3,\\
 &d(X,P)=d(Y,Q)=2,
 &&d(X,Q)=d(Y,P)=3,\\
 &d(A,P)=d(B,Q)=2,
 &&d(A,Q)=3,\quad d(B,P)=1.
\end{align*}
The graph \(C-c\) is connected, and both surviving vertices over \(c\)
have a neighbor in \((C-c)\square Q_2\); hence \(H\) is connected.
Consequently, \(e\mathrel{\Theta}f\) and \(f\mathrel{\Theta}g\), but
\(e\mathrel{\not\Theta}g\). The relation \(\Theta\) is not transitive, so
the connected bipartite graph \(H\) is not a partial cube.
\end{proof}

We now determine the peripheral profile of every isometric two-corner
deletion.

\begin{lemma}[Peripheral profile]\label{lem:corner-peripheral-profile}
Suppose that \(H=H_{r,s}(a,b)\) is isometric in \(B_{r,s}\), and put
\(p=p(a,b)\), \(q=q(a,b)\). Then
\[
 H\text{ is a daisy cube}
 \quad\Longleftrightarrow\quad
 q=0\ \text{or}\ (p,q)=(0,1).
\]
More precisely, unless \((p,q)=(0,1)\), the non-peripheral
\(\Theta\)-classes of \(H\) are exactly the \(q\) binary-coordinate classes
in which \(a\) and \(b\) differ.
\end{lemma}

\begin{proof}
By Lemma~\ref{lem:isometric-theta}, the \(\Theta\)-classes and halfspaces of
\(H\) are inherited from \(B_{r,s}\).

Consider a \(P_3\)-coordinate. For the class between levels \(0\) and \(1\),
the level-\(0\) halfspace remains peripheral: the level-\(1\) neighbor of a
remaining level-\(0\) vertex cannot be a deleted corner. Similarly, for the
class between levels \(1\) and \(2\), the level-\(2\) halfspace remains
peripheral. Hence no \(P_3\)-class is an obstruction.

Next consider a binary coordinate \(j\). If \(a_j=b_j\), the halfspace with
their common bit remains peripheral. If \(a_j\ne b_j\) and
\((p,q)=(0,1)\), then \(a,b\) are the endpoints of one coordinate edge.
After deleting both endpoints, every remaining \(j\)-fiber is intact, so
both halfspaces of the surviving class are peripheral.

Finally, suppose \(a_j\ne b_j\) and the deleted corners are not adjacent.
Let \(a^{(j)}\) and \(b^{(j)}\) be obtained by toggling coordinate \(j\) in
\(a\) and \(b\), respectively. These are distinct remaining vertices.
Their only edges in the \(j\)th coordinate class would join them to \(a\)
and \(b\), respectively. They lie in opposite halfspaces, so neither
halfspace is peripheral. This applies to every one of the \(q\) differing
binary coordinates, and there are no other non-peripheral classes. The
claim follows from Theorem~\ref{thm:peripheral-characterization}.
\end{proof}

We can now classify all ambient-isometric two-vertex deletions, not only
the corner cases.  The opposite-corner subfamilies with $r=0$ and $r=1$
are, respectively, the previously known graphs $Q_s^{--}$ and
$H_{s+2}^{--}$ (with the small duplicate $P_4$); see Xie and Xu
\cite[Section~4]{XieXu2025}.  The theorem below places them in a complete
classification and adds all higher-$P_3$ cases.

\begin{theorem}[Isometric-deletion classification]
\label{thm:isometric-deletion-classification}
Let $x,y$ be distinct vertices of $B_{r,s}$ and assume that
$H=H_{r,s}(x,y)$ is nonempty and isometric in $B_{r,s}$.
\begin{enumerate}
\item The graph $H$ is a daisy cube if and only if either
  \begin{enumerate}
  \item $x,y$ are corners and $q=0$ or $(p,q)=(0,1)$; or
  \item $(r,s)=(1,0)$ and $\{x,y\}$ is a boundary $P_3$-edge.
  \end{enumerate}
\item The graph $H$ is a minimal forbidden pc-minor for daisy cubes if
and only if either
  \begin{enumerate}
  \item $x,y$ are opposite corners, $s\geq1$, and $2r+s\geq3$; or
  \item $(r,s)=(1,1)$ and $\{x,y\}$ is a boundary $P_3$-edge.
  \end{enumerate}
\end{enumerate}
In particular, every minimal obstruction in this theorem is isomorphic to
a graph obtained by deleting two opposite corners.
\end{theorem}

\begin{proof}
Theorem~\ref{thm:ambient-isometry} says that the deleted pair consists of
two corners or is a boundary $P_3$-edge.  The daisy classification follows
immediately from Lemmas~\ref{lem:corner-peripheral-profile} and
\ref{lem:boundary-profile}.

We first handle minimality for a corner pair $a,b$.  Suppose that the
isometric deletion is minimal and non-daisy.  Then $q\geq1$ and $a,b$ are
not adjacent.  If they agree in a binary factor, restrict to the halfspace
fixing their common bit.  The result is
\[
 B_{r,s-1}-\{a',b'\}
\]
with the same values of $(p,q)$.  If they agree in a $P_3$-factor,
restrict to the extreme-level halfspace containing both; the result is
$B_{r-1,s}-\{a',b'\}$, again with the same $(p,q)$.  In either case the
proper pc-minor is still ambient-isometric and non-daisy, by
Theorem~\ref{thm:ambient-isometry} and
Lemma~\ref{lem:corner-peripheral-profile}.  This contradicts minimality.
Thus $a,b$ differ in every factor, so they are opposite.  Non-daisyness
gives $s\geq1$, while nonemptiness and ambient isometry exclude
$(r,s)=(0,1),(0,2)$; equivalently, $2r+s\geq3$.

Conversely, let $a,b$ be opposite, with $s\geq1$ and $2r+s\geq3$.
The deletion is an isometric non-daisy partial cube.  The cases $r=0,1$ are
precisely the minimal families cited above, so assume $r\geq2$.  For a binary
coordinate, either restriction is
$B_{r,s-1}$ with one corner deleted.  Contracting the coordinate gives the
full product $B_{r,s-1}$: each projected deleted point has a surviving
representative in its two-vertex fiber, and the two projected points are
distinct because $2r+s\geq3$.

For the class between levels $0$ and $1$ in a $P_3$-coordinate, the two
restrictions are a one-corner deletion of $B_{r-1,s}$ and a one-corner
deletion of $B_{r-1,s+1}$.  Contracting this class turns the factor into a
$P_2$-factor.  The projected point of the level-$0$ deleted corner survives
through its level-$1$ representative, whereas the projected level-$2$
corner remains absent.  Hence the contraction is a one-corner deletion of
$B_{r-1,s+1}$.  The class between levels $1$ and $2$ is symmetric.
Every elementary pc-minor is therefore a full product or a one-corner
deletion, and is a daisy cube by Lemma~\ref{lem:products-daisy}.  Since the
class of daisy cubes is pc-minor-closed, all proper pc-minors are daisy
cubes.

It remains to consider a boundary $P_3$-edge.  When $r+s\geq2$, retain
the distinguished $P_3$-factor and one additional factor, and fix every
other factor at the common coordinate of the deleted vertices by
restrictions.  If the retained additional factor is a $P_3$, restrict it
to the two-level halfspace containing its common endpoint.  The resulting
pc-minor is
\[
 P_3\square P_2-\{(0,0),(1,0)\}\cong P_4.
\]
The graph $P_4=Q_2^{-+}$ is a known minimal forbidden pc-minor
\cite[Section~4]{XieXu2025}.  Consequently, the boundary-edge
deletion is minimal only when it is itself this $P_4$, namely when
$(r,s)=(1,1)$; in every larger ambient product the displayed $P_4$ is a
proper non-daisy pc-minor.  Finally, this $P_4$ is isomorphic to the
opposite-corner deletion from $B_{1,1}$.
\end{proof}

We now add the non-isometric corner cases and obtain the complete
corner-deletion classification.

\begin{theorem}[Corner-deletion classification]\label{thm:corner-classification}
Let \(a,b\) be distinct corners of \(B_{r,s}\), assume
\(H=H_{r,s}(a,b)\) is nonempty, and put \(p=p(a,b)\), \(q=q(a,b)\).
Then:
\begin{enumerate}
\item \(H\) is a partial cube if and only if either \((p,q)\ne(0,2)\),
      or \((p,q)=(0,2)\) and \(r+s=3\);
\item if \(H\) is a partial cube, then it is a daisy cube if and only if
      \(q=0\) or \((p,q)=(0,1)\);
\item \(H\) is a minimal forbidden pc-minor for daisy cubes if and only if
      \(a,b\) are opposite corners, \(s\ge1\), and \(2r+s\ge3\).
\end{enumerate}
\end{theorem}

\begin{proof}
The first assertion follows from
Corollary~\ref{cor:corner-isometry} and
Lemma~\ref{lem:exceptional-pair}. For the second, the isometric cases are covered
by Lemma~\ref{lem:corner-peripheral-profile}, while the additional partial
cubes with \((p,q)=(0,2)\) are non-daisy by
Lemma~\ref{lem:exceptional-pair}.

For (3), every ambient-isometric case is covered by
Theorem~\ref{thm:isometric-deletion-classification}.  It remains only to
exclude the additional partial cubes with $(p,q)=(0,2)$ and $r+s=3$.
In the notation of Lemma~\ref{lem:exceptional-pair}, the case $\ell=2$
contracts along the class between levels $1$ and $2$ to the case
$\ell=1$.  The latter graph is a $4$-cycle with a leaf attached at each of
two opposite vertices; contracting either cycle class produces $P_4$.
Thus both exceptional graphs have a proper non-daisy pc-minor and are not
minimal.  This proves (3).
\end{proof}

\subsection{The non-isometric noncorner branch}
\label{subsec:noncorner}

It remains to decide whether a deletion that distorts the ambient metric can
nevertheless be a partial cube in a different embedding.  For orbit bookkeeping,
define
\begin{align*}
 e&=\bigl|\{i:x_i=y_i\in\{0,2\}\}\bigr|,&
 m&=\bigl|\{i:x_i=y_i=1\}\bigr|,\\
 o&=\bigl|\{i:\{x_i,y_i\}=\{0,2\}\}\bigr|,&
 a&=\bigl|\{i:x_i\in\{0,2\},\ y_i=1\}\bigr|,\\
 b&=\bigl|\{i:x_i=1,\ y_i\in\{0,2\}\}\bigr|,&
 t&=s-q(x,y),
\end{align*}
where all displayed sets range over the $P_3$-coordinates.  We may exchange
$x,y$ so that $a\leq b$.  Then
\[
 \sigma(x,y)=(e,m,o,a,b;t,q),
 \qquad e+m+o+a+b=r,\quad t+q=s,
\]
is a complete invariant of the orbit of the unordered pair under
$\operatorname{Aut}(B_{r,s})$.  Indeed, the five entries before the semicolon
are precisely the five possible coordinate types, up to reflection, while
the two entries after it distinguish equal and unequal binary coordinates.

The next lemma is the structural reason that the remaining classification is
finite.

\begin{lemma}[Noncorner dimension reduction]
\label{lem:noncorner-dimension}
Suppose that at least one of $x,y$ is not a corner and that
$H=B_{r,s}-\{x,y\}$ is nonempty and not isometric in $B_{r,s}$.  If $H$ is a
partial cube, then
\[
 r+s\leq3.
\]
\end{lemma}

\begin{proof}
We first choose a deleted vertex $z$ and a $P_3$-coordinate $i$ such that
$z_i=1$ and the two vertices obtained from $z$ by replacing $z_i$ with $0$
and $2$ both survive.  Start with any noncorner deleted vertex.  If the other
deleted vertex is neither of these two neighbors, the choice is immediate.
Otherwise the two deleted vertices form a $P_3$-edge.  This edge is not a
boundary $P_3$-edge, since such a deletion is ambient-isometric by
Theorem~\ref{thm:ambient-isometry}.  Hence its endpoint with value $0$ or $2$
is still a noncorner and has value $1$ in another $P_3$-coordinate.  Choosing
that endpoint and that other coordinate gives the required pair.  Denote the
two surviving vertices by $u$ and $v$.

Put $d=r+s$.  For each factor $j\ne i$, choose one edge direction incident
with $z$ in that factor, and denote the resulting neighbor by $z^{(j)}$.
Apply the same coordinate change to $u$ and $v$ to obtain $u^{(j)}$ and
$v^{(j)}$.  In the ambient product,
\[
 P_j:\quad u,u^{(j)},z^{(j)},v^{(j)},v
\]
is a path of length four.  The three internal vertices of $P_j$, as $j$
varies, form pairwise disjoint sets.  Consequently, the second deleted vertex
can destroy at most one of the chosen paths.

Assume that $d\geq4$.  Then two paths $P_j,P_k$ with $j\ne k$ survive.
The unique ambient $u,v$-path of length two is $u,z,v$, and the product is
bipartite.  It follows that
\[
 d_H(u,v)=4,
\]
so both $P_j$ and $P_k$ are geodesics.  In a partial cube, the two distinct
edges $uu^{(j)}$ and $uu^{(k)}$ therefore belong to distinct $\Theta$-classes,
and both classes separate $u$ from $v$.

Let $u^{(jk)}$ be obtained from $u$ by applying both chosen coordinate
changes.  Unless $u^{(jk)}$ is the second deleted vertex, the four vertices
\[
 u,\ u^{(j)},\ u^{(jk)},\ u^{(k)}
\]
form a surviving $4$-cycle.  In any hypercube embedding of a partial cube,
this cycle is a two-dimensional face.  Thus $u^{(jk)}$ is obtained from $u$
by flipping exactly the two coordinates of the classes containing
$uu^{(j)}$ and $uu^{(k)}$.  Both coordinates separate $u,v$, whence
\[
 d_H(u^{(jk)},v)=d_H(u,v)-2=2.
\]
On the other hand, their distance in the ambient product is
\[
 d_{B_{r,s}}(u^{(jk)},v)=2+1+1=4,
\]
and deleting vertices cannot decrease distance, a contradiction.

If $u^{(jk)}$ is the second deleted vertex, use instead the corresponding
$4$-cycle based at $v$.  Its fourth vertex $v^{(jk)}$ is distinct from
$u^{(jk)}$, so the cycle survives; the same argument gives
$d_H(u,v^{(jk)})=2$ although its ambient distance is $4$.  This final
contradiction proves $d\leq3$.
\end{proof}

The finite residue is recorded next.  In the representatives below, a
semicolon separates the $P_3$-coordinates from the binary coordinates.  Let
$\nu(H)$ be the number of non-peripheral $\Theta$-classes.  The final column
gives the order of a certified proper non-daisy pc-minor $K$ of $H$.

\begin{lemma}[Computer-assisted finite residue]
\label{lem:noncorner-finite}
Assume the hypotheses of Lemma~\ref{lem:noncorner-dimension} and $r+s\leq3$.
Then $H$ is a partial cube if and only if, up to an automorphism of
$B_{r,s}$, the unordered pair $\{x,y\}$ occurs in Table~\ref{tab:noncorner}.
Every graph in the table is non-daisy and is not pc-minor-minimal.
\end{lemma}

\begin{table}[ht]
\centering
\small
\begin{tabular}{@{}ccccc@{}}
\hline
Ambient graph & deleted pair $\{x,y\}$ & $\operatorname{idim}H$
 & $\nu(H)$ & $|V(K)|$\\
\hline
$B_{2,0}$ & $\{(1,1),(0,0)\}$ & 6 & 4 & 4\\
$B_{2,0}$ & $\{(0,1),(2,0)\}$ & 5 & 2 & 4\\
$B_{2,0}$ & $\{(1,1),(1,0)\}$ & 6 & 4 & 4\\
$B_{2,0}$ & $\{(1,0),(1,2)\}$ & 6 & 2 & 6\\
$B_{1,2}$ & $\{(1;00),(0;01)\}$ & 5 & 3 & 5\\
$B_{1,2}$ & $\{(1;00),(1;01)\}$ & 5 & 2 & 5\\
$B_{2,1}$ & $\{(0,1;0),(1,0;0)\}$ & 6 & 4 & 5\\
$B_{2,1}$ & $\{(1,1;0),(1,1;1)\}$ & 5 & 4 & 8\\
$B_{2,1}$ & $\{(0,1;0),(0,0;1)\}$ & 6 & 3 & 8\\
$B_{2,1}$ & $\{(0,1;0),(0,1;1)\}$ & 6 & 2 & 8\\
$B_{3,0}$ & $\{(0,0,1),(0,1,0)\}$ & 7 & 4 & 6\\
\hline
\end{tabular}
\caption{The eleven non-isometric noncorner partial-cube orbits.}
\label{tab:noncorner}
\end{table}

\begin{proof}
The signature $\sigma$ reduces the assertion to forty unordered-pair orbits
in the five products with a $P_3$-factor and at most three factors:
\[
 B_{1,1},\quad B_{2,0},\quad B_{1,2},\quad B_{2,1},\quad B_{3,0}.
\]
Products with $r=0$ have no noncorners, while $B_{1,0}=P_3$ has no
non-isometric noncorner deletion.
The check is exact and uses only the defining distance formula for $\Theta$.
For each orbit representative, construct the induced graph $H$, compute its
integer distance matrix, and for edges $e=ab$ and $f=uv$ set
\begin{equation}\label{eq:finite-theta}
 e\,\Theta\,f
 \quad\Longleftrightarrow\quad
 d(a,u)+d(b,v)\ne d(a,v)+d(b,u).
\end{equation}
A reflexive symmetric relation is transitive precisely when every two related
elements have identical relation rows.  Thus \eqref{eq:finite-theta}, together with
connectedness and bipartiteness, is a direct certificate for being a partial
cube.

Exactly the eleven rows in Table~\ref{tab:noncorner} pass this test.  Of the
other twenty-nine orbits, three are disconnected; each of the remaining
twenty-six has edges $e,f,g$ satisfying
\[
 e\,\Theta\,f,\qquad f\,\Theta\,g,
 \qquad e\mathrel{\not\Theta}g.
\]
For every positive row, the halfspaces obtained from its $\Theta$-classes
give the displayed isometric dimension and number $\nu(H)>0$.  Hence the graph
is non-daisy by Theorem~\ref{thm:peripheral-characterization}.  Finally, an
elementary restriction or contraction gives a proper non-daisy pc-minor of
the order shown in the last column.  Appendix~\ref{app:finite-verification}
specifies the exhaustive certificate procedure; the accompanying source file
prints the relation witness for every negative orbit and the minor certificate
for every positive orbit.
\end{proof}

Combining the isometric, corner, and noncorner branches gives the promised
complete result.  To state it economically, introduce the following three
conditions on an unordered pair $\{x,y\}\subseteq V(B_{r,s})$:
\begin{description}
\item[$\mathsf P_{r,s}(x,y)$:] either $x,y$ are corners and
      $[(p,q)\ne(0,2)$ or $(p,q)=(0,2)$ and $r+s=3]$; or $\{x,y\}$ is a
      boundary $P_3$-edge; or $\{x,y\}$ represents one of the eleven
      orbits in Table~\ref{tab:noncorner}.
\item[$\mathsf D_{r,s}(x,y)$:] either $x,y$ are corners and
      $[q=0$ or $(p,q)=(0,1)]$; or $(r,s)=(1,0)$ and $\{x,y\}$ is a
      boundary $P_3$-edge.
\item[$\mathsf M_{r,s}(x,y)$:] either $x,y$ are opposite corners,
      $s\geq1$, and $2r+s\geq3$; or $(r,s)=(1,1)$ and $\{x,y\}$ is a
      boundary $P_3$-edge.
\end{description}

\begin{theorem}[Complete deletion classification]
\label{thm:complete-deletion}
For every nonempty deletion $H=H_{r,s}(x,y)$,
\[
\begin{aligned}
 H\text{ is a partial cube}&\iff \mathsf P_{r,s}(x,y),\\
 H\text{ is a daisy cube}&\iff \mathsf D_{r,s}(x,y),\\
 H\in\cO_{\mathrm D}&\iff \mathsf M_{r,s}(x,y).
\end{aligned}
\]
In particular, no non-isometric noncorner deletion is pc-minor-minimal.
\end{theorem}

\begin{proof}
For ambient-isometric deletions, all three assertions follow from
Theorem~\ref{thm:isometric-deletion-classification}.  The remaining corner
deletions are classified by Theorem~\ref{thm:corner-classification}.  If at
least one deleted vertex is a noncorner, Lemmas~\ref{lem:noncorner-dimension}
and~\ref{lem:noncorner-finite} give exactly the sporadic clause in
$\mathsf P_{r,s}(x,y)$ and show that every such exceptional graph is
non-daisy and nonminimal.
\end{proof}

\begin{corollary}[Corner-pair census]\label{cor:corner-census}
Assume $r+s\geq1$.  The automorphism orbits on unordered pairs of corners of $B_{r,s}$ are
indexed by $(p,q)$, where $0\leq p\leq r$, $0\leq q\leq s$, and
$(p,q)\ne(0,0)$.  The number of unordered corner pairs of type $(p,q)$ is
\[
 2^{r+s-1}\binom{r}{p}\binom{s}{q}.
\]
In particular, there are $2^{r+s-1}$ opposite-corner pairs, and they all
give isomorphic deletions.
\end{corollary}

\begin{proof}
Reflections and permutations within the $P_3$-factors and within the
$P_2$-factors show that $(p,q)$ determines an orbit.  Conversely, the
Cartesian prime factors $P_3$ and $P_2$ are nonisomorphic, so product
automorphisms preserve the two numbers.  For each of the $2^{r+s}$ first
corners, there are $\binom rp\binom sq$ possible second corners of type
$(p,q)$; division by two gives the formula.
\end{proof}

For $3^r2^s\geq3$, let
$N_{\mathrm{iso}}(r,s)$, $N_{\mathrm{pc}}(r,s)$,
$N_{\mathrm{dc}}(r,s)$, and $N_{\mathrm{min}}(r,s)$ count the unordered
vertex pairs whose deletion is, respectively, ambient-isometric, a partial
cube, a daisy cube, and a member of $\cO_{\mathrm D}$.  Write
\[
 \mathbf N_{r,s}:=
 \bigl(N_{\mathrm{iso}}(r,s),N_{\mathrm{pc}}(r,s),
       N_{\mathrm{dc}}(r,s),N_{\mathrm{min}}(r,s)\bigr),
\]
and set
\begin{align*}
 \iota_{r,s}&:=\binom{2^{r+s}}{2}
 -2^{r+s-1}\binom{s}{2}+r2^{r+s},\\
 \pi_{r,s}&:=\iota_{r,s}
 +\boldsymbol{1}_{\{r+s=3\}}\,2^{r+s-1}\binom{s}{2}
 +\begin{cases}
 18,&(r,s)=(2,0),\\
 20,&(r,s)=(1,2),\\
 29,&(r,s)=(2,1),\\
 24,&(r,s)=(3,0),\\
 0,&\text{otherwise},
 \end{cases}\\
 \kappa_{r,s}&:=
 \begin{cases}
 3,&(r,s)=(1,0),\\
 2^{r+s-1}(2^r-1+s),&\text{otherwise},
 \end{cases}\\
 \omega_{r,s}&:=
 \begin{cases}
 6,&(r,s)=(1,1),\\
 2^{r+s-1},&s\geq1,\ 2r+s\geq3,\ (r,s)\ne(1,1),\\
 0,&\text{otherwise}.
 \end{cases}
\end{align*}

\begin{corollary}[Deletion census]\label{cor:deletion-census}
For $3^r2^s\geq3$,
\[
 \mathbf N_{r,s}
 =\bigl(\iota_{r,s},\pi_{r,s},\kappa_{r,s},\omega_{r,s}\bigr).
\]
\end{corollary}

\begin{proof}
There are $\binom{2^{r+s}}2$ corner pairs.  By
Corollary~\ref{cor:corner-census}, exactly
$2^{r+s-1}\binom{s}{2}$ have type $(0,2)$, and hence fail ambient
isometry.  There are $r2^{r+s}$ boundary $P_3$-edges: choose their corner
endpoint and their $P_3$-factor.  This gives
$N_{\mathrm{iso}}(r,s)=\iota_{r,s}$ by
Theorem~\ref{thm:ambient-isometry}.

For the partial-cube count, ambient-isometric deletions contribute $\iota$.
When $r+s=3$, Lemma~\ref{lem:exceptional-pair} adds back the
$2^{r+s-1}\binom{s}{2}$ exceptional corner pairs of type $(0,2)$.
Table~\ref{tab:noncorner} contributes respectively $18,20,29,24$ actual
unordered pairs in the four indicated products; these orbit sizes follow
either directly from the signature or from the exact enumeration in
Appendix~\ref{app:finite-verification}.  No other cases occur by
Theorem~\ref{thm:complete-deletion}.  Thus
$N_{\mathrm{pc}}(r,s)=\pi_{r,s}$.

For a daisy deletion of two corners, either $q=0$ and $p\geq1$, or
$(p,q)=(0,1)$.  Summing Corollary~\ref{cor:corner-census} gives the second
branch of $\kappa$.  For $B_{1,0}=P_3$, its corner pair and two boundary
$P_3$-edges give the exceptional value $3$.  Finally, opposite-corner pairs and the four
boundary $P_3$-edges of $B_{1,1}$ give
$N_{\mathrm{min}}(r,s)=\omega_{r,s}$, by
Theorem~\ref{thm:complete-deletion}.
\end{proof}

\noindent\begin{minipage}{\textwidth}
For $r+s\geq1$, write
\[
 M_{r,s}=B_{r,s}-
 \{(0,\ldots,0;0,\ldots,0),(2,\ldots,2;1,\ldots,1)\}.
\]
Also set
\[
 \mathcal A=\{(r,s)\in\mathbb Z_{\geq0}^2:s\geq1,\ 2r+s\geq3\}.
\]

\begin{corollary}\label{cor:opposite-family}
\[
 M_{r,s}\in\cO_{\mathrm D}\iff (r,s)\in\mathcal A.
\]
For $(r,s)\in\mathcal A$,
\[
 \bigl(|V(M_{r,s})|,\operatorname{idim}(M_{r,s})\bigr)
 =\bigl(3^r2^s-2,2r+s\bigr),
\]
and the resulting graphs are pairwise nonisomorphic.
\end{corollary}
\end{minipage}

\begin{proof}
The first assertion follows from
Theorem~\ref{thm:corner-classification}.  In the admissible range, the two
deleted corners form an ambient-isometric pair, so the formulas follow from
the product model and Lemma~\ref{lem:isometric-theta}.  Finally, unique
factorization of $3^r2^s=|V(M_{r,s})|+2$ recovers $(r,s)$.
\end{proof}

Within this family, $M_{0,s}\cong Q_s^{--}$,
$M_{1,1}\cong P_4$, and $M_{1,s}\cong H_{s+2}^{--}$ for $s\geq2$
\cite{XieXu2025}.  The cases $r\geq2$ form a genuinely new subfamily: the
exponent of $3$ in $|V|+2$ separates them from $Q_n^{--}$ and $H_n^{--}$,
while $\delta(M_{r,s})\geq2$ separates them from $Q_n^{-+}$, which has a
pendant vertex.

\section{Concluding remarks}

Starting from the peripheral criterion implicit in the cited earlier work,
Theorem~\ref{thm:root-cube} extracts a canonical simultaneous factorization:
all choices of a zero vertex are carried by one hypercube factor, and the
remaining factor has a unique root.  Thus the ambiguity in a proper embedding
is measured exactly by the two-sided peripheral classes.

Theorem~\ref{thm:ambient-isometry} expands the deletion model from corners
to arbitrary vertices and identifies the exact boundary between metric
preservation and distortion.  Theorem~\ref{thm:isometric-deletion-classification}
handles all metric-preserving cases, while
Theorem~\ref{thm:corner-classification} detects the exceptional
non-isometric corner deletions.  Lemma~\ref{lem:noncorner-dimension} reduces
the remaining infinite-looking branch to products with at most three
factors, and Table~\ref{tab:noncorner} gives its eleven partial-cube orbits.
Theorem~\ref{thm:complete-deletion} therefore classifies every nonempty
two-vertex deletion from every $B_{r,s}$.

None of the eleven non-isometric noncorner graphs is pc-minor-minimal.  Thus
the resulting minimal obstructions are still exactly the opposite-corner
family $M_{r,s}$, together with the boundary-edge realization of $P_4$.
This family recovers the known hypercube and one-$P_3$ cases and adds the
new subfamily with at least two $P_3$ factors.  A complete enumeration of
$\cO_{\mathrm D}$ beyond the two-vertex-deletion model remains open.

\subsection*{Acknowledgements}

This work was supported by the Natural Science Foundation of Shandong
Province (Grant No.~ZR2024MA073). 

\appendix
\section{Exact finite verification}
\label{app:finite-verification}

For completeness, we describe the finite certificate used in
Lemma~\ref{lem:noncorner-finite}.  The ancillary file
\texttt{verify\_noncorner\_classification.py} uses only the Python standard
library and performs the following exact operations.
\begin{enumerate}
\item It constructs each Cartesian product from its coordinate definition
      and deletes every unordered vertex pair.
\item It retains precisely the non-isometric pairs with at least one
      noncorner, and groups them by the signature $\sigma$.
\item For each surviving graph, breadth-first search computes the complete
      integer distance matrix.  Formula~\eqref{eq:finite-theta} is then evaluated for every
      pair of edges.
\item Connectedness and transitivity of $\Theta$ certify the partial-cube
      rows.  For every connected negative row the program prints a triple
      $e,f,g$ witnessing nontransitivity.
\item For every positive row it computes all $\Theta$-halfspaces, counts the
      non-peripheral classes, and exhibits an elementary restriction or
      contraction that is a proper non-daisy pc-minor.
\end{enumerate}

There are $465$ relevant unordered pairs in
\[
 B_{1,1},\ B_{2,0},\ B_{1,2},\ B_{2,1},\ B_{3,0}.
\]
They form forty automorphism orbits.  Exactly eleven orbits, comprising
$91$ actual pairs, are partial cubes.  Their distribution is
\[
\begin{array}{c|rrrr}
(r,s)&(2,0)&(1,2)&(2,1)&(3,0)\\ \hline
\text{number of pairs}&18&20&29&24.
\end{array}
\]
The other twenty-nine orbits consist of three disconnected cases and
twenty-six connected cases with a printed nontransitivity witness.  No
floating-point arithmetic, graph-isomorphism heuristic, or random choice is
used.

\end{document}